\documentclass{amsart}

\usepackage[a4paper, margin=1in]{geometry}

\usepackage{graphicx,amssymb,mathrsfs,amsmath,amsfonts,appendix}
\usepackage[colorinlistoftodos]{todonotes}
\usepackage{subfig}
\usepackage{enumitem}
\usepackage{natbib}
\usepackage{hyperref}
 \usepackage[foot]{amsaddr}

\newtheorem{thm}{Theorem}[section]
\newtheorem{definition}[thm]{Definition}

\newtheorem{lemma}[thm]{Lemma}

\newcommand{\M}{{\mathcal M}}

\let\pa\partial

\def\eps{\varepsilon}
\def\d{\,{\rm d}}

\title{Mean field models for segregation dynamics}

\author{Martin Burger$^1$}
\address{$^1$ WWU M{\"u}nster, Fachbereich Mathematik, Einsteinstr. 62, Germany}

\author{Jan-Frederik Pietschmann$^2$}
\address{$^2$ Technische Universit\"at Chemnitz, Reichenhainer Stra\ss{}e 41, Germany}
  
\author{Helene Ranetbauer$^3$}
\author{Christian Schmeiser$^3$}

\address{$^3$ University of Vienna, Oskar-Morgenstern-Platz 1, Austria}

\author{Marie-Therese Wolfram$^4$}
\address{$^4$ University of Warwick, Coventry CV4 7AL, UK and \\
  RICAM, Altenbergerstr. 69, 4040 Linz, Austria and \\
TU Munich, Boltzmannstr. 3, 85748 Garching, Germany}

\begin{document}

\label{firstpage}
\maketitle

\begin{abstract}
  In this paper we derive and analyse mean-field models for the dynamics of groups of individuals undergoing a random walk. The random motion of individuals is only influenced by the perceived densities of the different groups present as well as the available space. All individuals have the tendency to stay within their own group and avoid the others. These interactions lead to the formation of aggregates in case of a single species, and to segregation in the case of multiple species. We derive two different mean-field models, which are based on these interactions and weigh local and non-local effects differently. We discuss existence and stability properties of solutions for both models and illustrate the rich dynamics with numerical simulations.
\end{abstract}

\section{Introduction}

In 1969 the American economist Thomas Schelling postulated that if individuals belonging to one of two groups have an (arbitrary small) preference for their own group, the groups segregate,  see \cite{schelling1969models}. His claim was supported by an agent based model, in which individuals from two groups move randomly on a discrete lattice. In these simulations agents decrease their transition probabilities if they are surrounded by a certain fraction of group members, otherwise they move to any available site. They also indicate that the discrete system converges to a stationary state with aggregated and segregated states. The form of these complex stationary states depends on the preference for the own group, the initial distribution of the agents, and the occupancy of the domain.\\

In this paper we propose and analyse two mean-field models, which are inspired by the Schelling dynamics. In the considered models individuals move randomly - their random motion is only influenced by the perceived density and the available physical space. We start by stating a general mean-field model for both groups, in which the transition rates as well as the diffusivities depend on the nonlocal perceived density (via a convolution kernel). Then we formally discuss two different scalings. In the first case individuals sense the densities of the own and other group in a large surrounding, but are only allowed to move locally. This corresponds to a particular scaling of the convolution kernel in the transition rates and yields a partial differential equation (PDE) with non-local diffusivity in the formal limit. In the second case individuals can move to any available site in the domain, but their transition rate depends on the locally sensed group densities only. Here we rescale the convolution kernel in the diffusivities, and obtain an integro differential equation in the limit. We refer to the first situation as \emph{non-local sensing and local jumps}, to the second as \emph{local sensing and non-local jumps} throughout this paper. \\
We discuss the existence of solutions for both models and analyse the stability of stationary states. In both models we observe the expected formation of aggregates in the single species model and segregated states for two species. The characteristic shape of these non-trivial stationary states depends on the specific interaction rules, the occupancy of the domain and the diffusivity.\\

Segregation dynamics have been observed in many mean-field models of single and multi species interacting particle systems. Turing instabilities in linear reaction diffusion systems are among the most prominent examples, see \cite{Turing37}. Here already a large disparity in the diffusion coefficients may lead to the formation of segregated states. Chemotaxis models, in which individuals are attracted by the gradient of a chemical substrate, are another prominent example for aggregation dynamics. In this case the attraction towards the chemical substrate may result in the blow up of solutions or the formation of complex stationary states, see \cite{BDSS2008}. These dynamics become even more complex for nonlinear diffusion. In the case of a single species non-linear diffusion (depending locally or non-locally on the density) may lead to the formation of aggregates, see for example \cite{burger2013individual,Anguige2008}. Here the diffusivity may become degenerate or even negative for certain parameter and density regimes, leading to ill-posed problems.\\
In multiple species problems additional effects such as cross diffusion may initiate or enhance aggregation or segregation dynamics. For example \cite{bertsch2010} showed that in a multi-species systems with porous medium type diffusion, populations remain separated, if they were separated initially. Also cross-diffusion, which arises in many systems describing interactions among different species can lead to segregation. 
Cross diffusion may be caused by finite size effects (see \cite{burger2010nonlinear, burger2016lane, bruna2012diffusion}), alignment (see \cite{degond2017new, zhang2011general}) or attractive and repulsive interactions (see \cite{canizo2010collective, tao2013competing, carrillo2018zoology, carrillo2017splitting,burger2016lane,BFFS2018, BFH2014}). In several of these models coarsening dynamics of clusters and segregated states can be observed in numerical simulations. In some models it was even possible to study these coarsening dynamics rigorously, see for example \cite{BDSS2008} or \cite{PSTV2011}. 

With appropriate choices of the ingredients, the models considered here are applicable to various situations in cell biology. An example is cell segregation dynamics (\cite{Batlle}), important in embryonic development and cancer prevention, where sensing is a local process. Nonlocal sensing in cell populations is typically due to chemical signals, which can also lead to segregation with bacterial colonies as an example (\cite{BenAmar,PSTV2011}). 
In this processes local jump models would typically be used, but anomalous motility patterns involving jumps are observed in several cell types (\cite{Dieterich}).

This paper is organised as follows: we start by deriving the two different types of mean-field equations in Section \ref{s:model}. Section \ref{sec:analysis} and Section \ref{sec:stationarystates} study the existence and long time behaviour of solutions for both models. We conclude by illustrating the stability results as well as the dynamics with various numerical examples in Section \ref{sec:numerics}.

\section{Derivation of the mean-field models}\label{s:model}

We start by introducing two mean-field models both describing the dynamics of two interacting species, which we shall refer to as red and blue ones throughout this paper.
Individuals from either group move randomly in space, only influenced by the sensed density and the physically available space. We recall that we consider two different types of sensing mechanisms and transition rates:
\begin{enumerate}
\item \textit{Local jumps and non-local sensing:} Particles are allowed to move locally; their transition rates depend on the specific non-local density and the available space.
 \item \textit{Non-local jumps and local sensing:} Individuals can move in the entire domain; their transition rates depend on the local density and the available space.
\end{enumerate}
We will formally derive the two different mean-field models for both cases. In the first setting we obtain a system of nonlinear diffusion equations with nonlocal density dependence of the diffusion coefficients. In the second case we derive a nonlinear integro differential equation. \\

\subsection{Preliminaries}
Inspired by the interaction rules described above, we consider the equation
\begin{align}\label{eq:mf}
\begin{aligned}
\partial_t c=\int_{\mathbb{R}^N} K_1(x-x')[(1-\rho)D_c' c'-(1-\rho') D_c c]dx',
\end{aligned}
\end{align}
where $c=r,b$, $\rho=r+b$ and $r=r(x,t),\, b=b(x,t)$ represent the probability of finding a red or blue particle in location $x$ at time $t$. Here the dash indicates the evaluation at $x'$, for instance $c'=c(x',t)$. Moreover, the diffusion coefficient is given by
$$D_c(x,t)=D_c((K_2*r(\cdot,t))(x),(K_2*b(\cdot,t))(x)), $$
with convolutions of the form
$$ (K_2*u(\cdot,t))(x) = \int_{\mathbb{R}^N} K_2(x-x') u(x',t)~dx'.$$
Here the functions $K_1,K_2$ correspond to appropriate interaction kernels, whose properties we will specify later.\\
This general continuum model includes the following considerations:
\begin{enumerate}[label=(\roman*)]
\item The availability of physical space via the factor $(1-\rho)$: Since $\rho$ denotes the total density and $1$ is the maximum density, individuals can only move to a position if it is not fully occupied.
\item The preference to stay close to the own group: We assume that the diffusivities $D_{c}$ are non-increasing with respect to the own species and non-decreasing with respect to the other. Therefore $D_r(\cdot, \cdot)$ is non-increasing in the first argument and non-decreasing in the second one. Obviously the opposite holds true for the function $D_b(\cdot,\cdot)$.
\item Local and non-local effects via the kernels $K_{i}$: The kernels $K_{i}$, $i=1,2$ are positive, radially symmetric and non-increasing functions. We assume that it is more expensive to move further away, which is included via the kernel $K_1$. The kernel $K_2$ accounts for the fact that the transition rate is stronger influenced by the local density than the density far away.
\end{enumerate}

Equation \eqref{eq:mf} is a general model for random motion of individuals with density dependent diffusivities. Such mean field models have been proposed and studied in the context of aggregation dynamics - for example as already mentioned in cell segregation dynamics or in the collective motion of cockroaches, see \cite{burger2013individual}. However in the latter case the mean field equation was derived from a density dependent random walk (with no size exclusion). Other applications include mean field models for pedestrian dynamics, which can be derived from a discrete lattice based hopping approach, cf. \cite{burger2016lane}.

In the following we formally derive the limiting equations in the case of local jumps and non-local sensing as well as the case of non-local jumps and local sensing. The limiting equations are obtained by localising either $K_1$ or $K_2$ and performing a formal linearization. 

\subsection{Local jumps and non-local sensing}
In the first model individuals only move to their immediate neighbourhood, hence we assume $K_1$ to be of the form
\begin{align}
K_1(x)&=\frac{1}{\epsilon^{N+2}}\tilde{K}\left( \frac{x}{\epsilon} \right)=\frac{1}{\epsilon^{N+2}}\tilde{k}\left(\left|\frac{x}{\epsilon}\right| \right),
\end{align}
with $\tilde{k}:[0,\infty) \mapsto [0,\infty)$ and $\epsilon > 0$. This scaling ensures that the second moment is independent of $\epsilon$ and we choose the following normalisation
\begin{align}\label{eq:normsecondmoment}
\int_{\mathbb{R}^N} \tilde{K}(z)|z|^2\;dz = \left|\mathcal{S}^{N-1}\right|\int_0^\infty \tilde{k}(r)r^{N+1}\,dr=2N,
\end{align} with $N$ being the space dimension and where $|\mathcal{S}^{N-1}|$ denotes the surface area of the $(N-1)$-dimensional unit sphere. Note that the particular choice of $K_1$ corresponds to local dynamics if $\epsilon \ll 1$. Using the rescaled kernel in equation \eqref{eq:mf} and the change of variables $\frac{x-x'}{\epsilon}=z$ yields
\begin{align*}
\partial_t c(x,t)=\frac{1}{\epsilon^2}\int_{\mathbb{R}^N} \tilde{K}(z)[(1-\rho(x,t))&D_c(x-\epsilon z,t) c(x-\epsilon z,t)\\
&-(1-\rho(x-\epsilon z,t)) D_c(x,t) c(x,t)]dz.
\end{align*}
Then a formal Taylor expansion in $\epsilon$ around $x$ gives
\begin{align*}
\partial_t c(x,t)=\frac{1}{2}\int_{\mathbb{R}^N} \tilde{K}(z)[(1-\rho(x,t))&z^{T}\nabla ^2 (D_c(x,t) c(x,t))z\\
&+D_c(x,t) c(x,t) z^T \nabla^2 \rho(x,t)z ]dz + O(\epsilon),
\end{align*}
where the first order terms cancelled due to the radial symmetry of $K$. This also implies that
\begin{align*}
\int_{\mathbb{R}^N} \tilde{K}(z) z_i z_j \,dz&= \delta_{i,j} \int_{\mathbb{R}^N} \tilde{K}(z) z_i^2 \, dz = \delta_{i,j} \frac{1}{N} \int_{\mathbb{R}^N} \tilde{K}(z) | z|^2 \, dz\\
&=\delta_{i,j} \frac{1}{N} \left|\mathcal{S}^{N-1}\right| \int_0^\infty \tilde{k}(r) |r|^{N+1} \, dr.
\end{align*}
Using \eqref{eq:normsecondmoment} and neglecting higher order terms in $\epsilon$
we obtain
\begin{align*}
\partial_t c&=(1-\rho)\Delta (D_c c(x,t))+D_c c\Delta\rho.
\end{align*}
Hence the full system, written in divergence form, reads as
\begin{align}\label{system1}
\begin{aligned}
\pa_t r&=\nabla \cdot [(1-\rho)\nabla(D_r(K_2\ast r,K_2\ast b)r)+r D_r(K_2\ast r,K_2\ast b)\nabla \rho]\\
\pa_t b &= \nabla \cdot [(1-\rho)\nabla (D_b(K_2\ast r,K_2\ast b)b)+b D_b(K_2\ast r,K_2\ast b)\nabla \rho],
\end{aligned}
\end{align}
where $K_2$ corresponds to the non-local sensing kernel. We recall that individuals have a preference for the own group.  Hence the diffusivities $D_r$ and $D_b$ are
non-increasing functions with respect to the own species. So the dynamics of $r$ and $b$ are driven by the non-local diffusion and the physically available space.\\
We will also analyse the corresponding single species model, obtained by setting $b=0$ in \eqref{system1}, later on. Here the diffusivity of the single species decreases
with the perceived density. Hence we expect the formation of aggregates. A similar single species model was proposed and analysed by \cite{burger2013individual}. However, since this model
does not include finite volume effects, measure valued steady states are possible. 

\subsection{Non-local jumps and local sensing}
In the case of non-local jumps and local sensing we assume that now $K_2$ is of the form
\begin{align}\label{e:k2}
K_2(x)&=\frac{1}{\epsilon^N}\tilde{K}\left(\frac{x}{\epsilon}\right),
\end{align}
and, in addition, that the mass of $\tilde{K}$ is normalised, i.e. $\int_{\mathbb{R}^N} \tilde{K}(x)\, \d x=1$. This choice, together with a change of variables to $\frac{x-x'}{\epsilon}=z$ and linearization of \eqref{eq:mf} in $\epsilon$ around x gives, again neglecting higher order terms, 
\begin{align*}
\partial_t c=\int_{\mathbb{R}^N} K_1(x-x')[(1-\rho)&D_c' c'-(1-\rho') D_c c]\,dx',
\end{align*}
where $D_c=D_c(r,b)$.
Then the full system reads as 
\begin{align}\label{systemnonlocal4}
\begin{aligned}
\partial_t r&=\int_{\mathbb{R}^N} K_1(x-x')[(1-\rho)D_r' r'-(1-\rho') D_r r]\,dx'\\
\partial_t b&=\int_{\mathbb{R}^N} K_1(x-x')[(1-\rho)D_b' b'-(1-\rho') D_b b]\,dx'\\
\end{aligned}
\end{align}

Note that system \eqref{systemnonlocal4} is a nonlinear integro-differential system. The evolution of the densities is influenced by the local density via $D_r$ and $D_b$. However, individuals can move to any available site in the domain with a rate depending on $K_1$, only. We will perform numerical simulations of the corresponding single species model, obtained by setting $b=0$ in \eqref{systemnonlocal4}. Again we will observe the formation of expected aggregated states.

\section{Analysis of the mean-field models}\label{sec:analysis}
In this section, we present global existence results for both models \eqref{system1} and \eqref{systemnonlocal4}. The result for system \eqref{system1} follows the idea presented in Thm 4.1  in \cite{berendsen2017cross} and the result for system \eqref{systemnonlocal4} is based a Picard Lindel\"of type theorem in Banach spaces. \\
From now on, we analyse both models \eqref{system1} and \eqref{systemnonlocal4} in the case of periodic boundary conditions. Hence, we choose the domain $\Omega$ to be the $N-$dimensional torus, i.e. $\Omega=\mathbb{T}^N$. Note that as the convolution of any function with a periodic function is again periodic, the occurring convolutions in the models are well-defined. Furthermore, for the rest of the paper we consider diffusion coefficients of the following form only
\begin{align}\label{diff_coeff}
D_r(p,q)=C_re^{-C_{rr} p+C_{rb} q},\text{ and } \quad D_b(p,q)=C_be^{C_{br}p-C_{bb}q},
\end{align}
with constants $C_r,C_b,C_{rr}, C_{rb}, C_{br}, C_{bb} > 0$. 
In \eqref{systemnonlocal4}  we will choose $(p,q)=(r,b)$, in \eqref{system1} $(p,q)=(K_2\ast r, K_2\ast b)$.

Since $r, b$ and $\rho$ represent densities, we introduce the set
\begin{align*}
\mathcal{M}=\{(r,\, b) \in L^2(\mathbb{T}^N)^2 : 0 < r,\, b; \, r+b=\rho < 1\text{ a.e.} \},
\end{align*}
and define the class of admissible convolution kernels as follows.
\begin{definition}[Admissible Kernel]\label{def:admissible} We say that a kernel $K$ is admissible if the following conditions are satisfied:
\begin{itemize}
 \item[\textbf{(K1)}] $K \in W^{1,1}(\mathbb{R}^N)$,
 \item[\textbf{(K2)}] $K$ is positive and radially symmetric, i.e. $K(x)=k(|x|)$ and $k$ is non-increasing,
 \item[\textbf{(K3)}] $k(|x|)$ behaves at most as singular as the Coulomb kernel as $|x|\to 0$.
\end{itemize}
\end{definition}

We consider system \eqref{system1} and \eqref{systemnonlocal4} 
with initial data $(r_I,\,b_I) \in\mathcal{ \overline M}$, i.e.
$$
r(x,0) = r_I(x)\text{ and }b(x,0) = b_I(x).
$$ 
Note that the constants 
\begin{align}\label{e:stationary}
r_0 := \frac{\int_{\mathbb{T}^N} r_I(x) dx}{\lvert \mathbb{T}^N \rvert} \text{ and } b_0 := \frac{\int_{\mathbb{T}^N} b_I(x) dx}{\lvert \mathbb{T}^N \rvert} 
\end{align}
are stationary solutions of system \eqref{system1} and \eqref{systemnonlocal4} and that both systems are conservative, that is  $\frac{d}{dt}\int_{\mathbb{T}^N}  r \, dx=\frac{d}{dt}\int_{\mathbb{T}^N} b\, dx = 0$.\medskip

\subsection{Local jumps and non-local sensing}

The existence argument for system \eqref{system1} follows the lines of a proof for a similar system studied in \cite{berendsen2017cross}. Indeed, due to the special choice of $D_r$ and $D_b$ in \eqref{diff_coeff} the system can be rewritten as 
\begin{align}\label{eq:alternative}
\begin{aligned}
\pa_t r&=\nabla \cdot \left[D_r(K_2\ast r, K_2\ast b)\left((1-\rho)\nabla r + r \nabla \rho + r(1-\rho)\nabla(-C_{rr}K_2\ast r + C_{rb}K_2\ast b) \right)\right]\\
\pa_t b&=\nabla \cdot \left[D_b(K_2\ast r, K_2\ast b)\left((1-\rho)\nabla b + b \nabla \rho + b(1-\rho)\nabla(C_{br}K_2\ast r - C_{bb}K_2\ast b) \right)\right].
\end{aligned}
\end{align}
In this form, the equation has a drift-diffusion structure, where the cross-diffusion terms are exactly the same as in \cite{Burger2010,berendsen2017cross}. However system \eqref{eq:alternative} has a different mobility (due to the multiplication with $D_r$ and $D_b$).

This system can be interpreted as a formal gradient flow structure with respect to a Wasserstein type metric, see \cite{Otto2001} for more details. The respective energy functional is 
\begin{align}\label{eq:entropy}
\begin{aligned}
E(r,b) &= \int_{\mathbb{T}^N} r\log r + b\log b + (1-\rho)\log (1-\rho)\\
&+ r(-C_{rr}K_2\ast r + C_{rb}K_2\ast b) + b(C_{br}K_2\ast r - C_{bb}K_2\ast b)\, dx
\end{aligned}
\end{align}
and the mobility matrix 
\[M(r,b)=\begin{pmatrix}
C_re^{-C_{rr} K_2\ast r+C_{rb}K_2\ast b}\, r(1-\rho) & 0 \\
0 & C_be^{C_{br} K_2\ast r-C_{bb}K_2\ast b}\, b(1-\rho)
\end{pmatrix}.\]
Hence system \eqref{eq:alternative} in formal gradient flow structure is given as
\begin{align*}
 \left(\begin{array}{c}
      \partial_t r \\
      \partial_t b
    \end{array}\right) 
= \nabla\cdot \left( M(r,b) \nabla \left(\begin{array}{c}
      \partial_r E(r,b) \\
      \partial_b E(r,b)
    \end{array}\right)\right).
\end{align*}
Note that with Definition \ref{def:admissible} of the admissible kernels, we can guarantee that there exists at least one minimizer of the energy functional \eqref{eq:entropy}, see Thm 2.5 in \cite{berendsen2017cross}. We use the local part of the entropy functional to define the so called entropy variables $u$ and $v$  as
\begin{align}\label{eq:entropyvariables}
u := \log r - \log ( 1-\rho),\; \text{ and }\; v := \log b - \log ( 1-\rho).
\end{align}
Inverting these relations yields the priori bounds $0 \le r,b$ and $r+b\le 1$, which are a crucial ingredient of the proof since no maximum principle is available. This is often called the boundedness-by-entropy principle, see \cite{Burger2010,jungel2015boundedness}. Together with bounds obtained from the entropy dissipation, this is enough to prove the following theorem:

\begin{thm}\label{thm:pdebnd}
Let $T>0$, $D_{r,b}$ given by \eqref{diff_coeff} and let $K_2$ denote an admissible kernel in the sense of Definition \ref{def:admissible}. Consider the PDE system \eqref{eq:alternative} on $\mathbb{T}^N$ with 
initial conditions 
\begin{align*}
r(x,\, 0)=r_I(x) \quad \text{and} \quad b(x,\, 0)=b_I(x),\quad \text{for a.e. }x \in {\mathbb{T}^N},
\end{align*}
with $(r_I,b_I)\in\overline{\mathcal{M}}$ and with periodic boundary conditions. Then there exists a weak solution $(r,\, b)$ in 
$$
W=(L^2((0,T),\, L^2({\mathbb{T}^N}))\cap H^1((0,\, T),\, H^{-1}({\mathbb{T}^N})))^2
$$
such that additionally 
$$\rho,\, \sqrt{1-\rho}r,\, \sqrt{1-\rho}b \in L^2((0,\,T),\, H^1({\mathbb{T}^N}))$$
and furthermore $(r,b)\in \overline{\mathcal{M}}$ a.e. in [0,T].
\end{thm}

\begin{proof}
Most of the proof is almost verbatim to the one in of Theorem 4.1 in \cite{berendsen2017cross}. In fact, the only differences are the different signs of the non-local interaction terms and the modified mobility. However, since the interaction terms only need to be bounded in the appropriate spaces and the modification to the mobility is strictly positive on $\M$, these changes do not affect the proof. For completeness, we sketch the procedure: Using the definition of the entropy variables in \eqref{eq:entropyvariables}, we can rewrite equation \eqref{eq:alternative} as
\begin{align*}
\partial_t r &= \nabla \cdot [ D_r(K_2\ast r, K_2\ast b)(r(1-\rho)\nabla u + r(1-\rho)(-C_{rr}\nabla K_2\ast r+C_{rb}\nabla K_2\ast b))] \\
\partial_t b &= \nabla \cdot [ D_b(K_2\ast r, K_2\ast b)(b(1-\rho)\nabla v + b(1-\rho)(C_{br}\nabla K_2\ast r-C_{bb}\nabla K_2\ast b))].
\end{align*}
In this formulation it becomes clear that the (respective) first terms of the right hand side will drive the dissipation of the entropy. The convection terms on the other hand can be estimated using the smoothing properties of the convolutions and yield a linear growth term in the entropy. Indeed, a formal calculation shows that 
\begin{align*}
E(r,b) + &\frac{1}{4} \int_0^T\int_{{\mathbb{T}^N}} (1-\rho)|\nabla\sqrt{r}|^2 + (1-\rho)|\nabla\sqrt{b}|^2 + |\nabla\sqrt{1-\rho}|^2 + 2|\nabla \rho|^2 \,dx dt  \nonumber \\
&\quad \leq E(r_I,\, b_I) + CT.
\end{align*}
To use this a-priori estimate in a rigorous way, the system is approximated by an implicit in time discretization and subsequentially regularised. In particular we denote by  $\tau > 0$  the discrete time step, and consider the following time discrete problem

\begin{align*} 
  \frac{1}{\tau}
  \begin{pmatrix}
r_{k-1}-r_k\\
 b_{k-1}-b_k
\end{pmatrix}
  &=
\begin{pmatrix}
    \nabla \cdot [ D_r(r_k(1-\rho_k)\nabla u_k + r_k(1-\rho_k)(-C_{rr}\nabla K_2\ast r_k+C_{rb}\nabla K_2\ast b_k )) ] \\
\nabla \cdot [ D_b(b_k(1-\rho_k)\nabla v_k + b_k(1-\rho_k)(C_{br}\nabla K_2\ast r_k-C_{bb}\nabla K_2\ast b_k )) ]
\end{pmatrix} \\
&+\tau
\begin{pmatrix}
\Delta u_{k}-u_k\\
 \Delta v_k-v_k\\
\end{pmatrix}.
\end{align*}
This system is still nonlinear and existence of the discrete iterates is established by a fixed point argument. Finally, the (time discrete analogue) of the dissipation of the entropy functional \eqref{eq:entropy} yields a-priori bounds which are sufficient to pass to the limit $\tau\to 0$ and obtain existence of a weak solution to \eqref{eq:alternative}. 
\end{proof}

\subsection{Non-local jumps and local sensing}
Next we discuss global in time existence of the nonlinear integro differential equation \eqref{systemnonlocal4}. Local in time existence follows from Picard Lindel\"of, which can be extended to all times $T > 0$.
\begin{lemma}[Local existence]\label{lem:local} For every $(r_I,b_I) \in \overline{\mathcal{M}}$, $D_{r,b}$ given by \eqref{diff_coeff} and admissible $K_1$, there exists a positive $T>0$ and functions 
$$
(r,b)\in [C^1((0,T],L^\infty({\mathbb{T}^N}))]^2,
$$

which are unique solutions to \eqref{systemnonlocal4}.
\end{lemma}
\begin{proof}
Taking $(r_1,b_1)$ and $(r_2,b_2)$ in $\overline{\mathcal{M}}$ we can estimate the right hand side of the first equation in \eqref{systemnonlocal4} as follows
\begin{align}\label{lipschitz}\nonumber
&\|(1-\rho_1)K_1\ast (D_r(r_1,b_1)r_1) - D_r(r_1,b_1)r_1 K_1\ast(1-\rho_1)\\
\quad& - (1-\rho_2)K_1\ast (D_r(r_2,b_2)r_2) - D_r(r_2,b_2)r_2 K_1\ast(1-\rho_2)\|_{L^\infty({\mathbb{T}^N})}\\\nonumber
\qquad &\le \| K_1 \ast (D_r(r_1,b_1)r_1-D_r(r_2,b_2)r_2) -  (D_r(r_1,b_1)r_1-D_r(r_2,b_2)r_2)K_1\ast 1\|_{L^\infty({\mathbb{T}^N})}\\\nonumber
\qquad & \le C \|r_1-r_2\|_{L^\infty({\mathbb{T}^N})},
\end{align}
where we have used Young's inequality for convolutions and the constant only depends on the integral of $K_1$ as well as on
$$
\sup_{(r,b)\in \overline{\mathcal{M}}} D_r(r,b),\text{ and } \sup_{(r,b)\in \overline{\mathcal{M}}} D_b(r,b).
$$
Performing the same estimate on the second equation in \eqref{systemnonlocal4}, we conclude that the right hand side is Lipschitz continuous in $\overline{\mathcal {M}} \subset [L^\infty({\mathbb{T}^N})]^2$. Thus applying a version of Picard-Lindel\"of in this Banach space, \cite[Thm 3.2]{Deimling1977}, concludes the proof.
\end{proof}
We proceed by showing that the local solution obtained in the previous lemma remains in $\overline{\mathcal{M}}$ for all times. To this end, we define
\begin{align}\label{eq:defF}
F(r,b) := \binom{(1-\rho) K_1*(D_r r)-D_r(r, b) rK_1*(1-\rho)}{(1-\rho)   K_1*(D_b b) - D_b(r, b) b K_1*(1-\rho)}.
\end{align}
Following \cite[Thm 5.1]{Deimling1977}, we have to show that 
\begin{align}\label{eq:condinv}
F(r,b) \cdot \nu \le 0,
\end{align}
for all $(r,b)\in \partial\mathcal{M}$ and all vectors $\nu\in \mathcal{N}(r,b)$ where $\mathcal{N}(r,b)$ denotes the normal cone at the point $(r,b)$ defined as
\begin{align}\label{eq:normalcone}
\mathcal{N}(r,b) = \left\{ u \in L^2({\mathbb{T}^N})^2\; \left|\; \sup \left\langle \mathcal{M}-\binom{r}{b},u\right\rangle \le 0\right\}\right..
\end{align}
First we note that $ \overline{\mathcal{M}}$ has empty interior and thus every element of $ \overline{\mathcal{M}}$ is an element of its boundary. First we consider all functions $(r,b)\in  \overline{\mathcal{M}}$ for which there exists an $\eps > 0$ such that $0< \eps \le r,b$ and $r+b \le 1- \eps < 1$. In this case, the normal only contains the vector $(0,0)$ and \eqref{eq:condinv} is trivially satisfied. The remaining parts of $\partial\mathcal{M}$ are of the form that, for a given set $A$ with positive Lebesgue measure, either 
\begin{align}\label{eq:charbdryM}
r = 0, \; b = 0 \text{ or } r+b = 1\text{ for a.e.  }x \in A.
\end{align}
In the first case, it is easy to check that all elements of the normal cone are of the form
$$
\nu_1 = \binom{-c}{0}\text{ on $A$ and zero otherwise,}
$$
for arbitrary constants $c>0$. Using $r=0$ in \eqref{eq:defF}, we obtain
$$
F(r,b) \cdot \nu_1 = 0,
$$
so that \eqref{eq:condinv} is also fulfilled. The same reasoning applies for the remaining cases in \eqref{eq:charbdryM} where we obtain the vectors
$$
\nu_2 = \binom{0}{-c}\text{ and } \nu_3 = \binom{c}{c}\text{ on }A\text{ and zero otherwise.}
$$
All remaining point in $\partial\mathcal{M}$ are combinations of the cases given in \eqref{eq:charbdryM} and thus we conclude that condition \eqref{eq:condinv} holds for all vectors in the normal cone and for each point in $\partial\mathcal{M}$. Then, \cite[Thm 5.1]{Deimling1977} ensures that, for every $(r_0,b_0) \in  \overline{\mathcal{M}}$, the corresponding solutions $(r,b)$ to \eqref{systemnonlocal4} remain in $ \overline{\mathcal{M}}$. In particular, this implies that the Lipschitz estimate \eqref{lipschitz} (which relies on the fact that $(r,b)\in\overline{\mathcal{M}}$) holds uniformly and thus \cite[Thm 3.4]{Deimling1977} yields
\begin{thm} For every $(r_I,b_I) \in  \overline{\mathcal{M}}$, $D_{r,b}$ given by \eqref{diff_coeff} and $K_1$ admissible, the unique local solutions to \eqref{systemnonlocal4} satisfy $(r,b) \in  \overline{\mathcal{M}}$ and furthermore exist for arbitrary times $T>0$.
\end{thm}

\section{Linear stability and nontrivial steady states}\label{sec:stationarystates}

In this section we discuss stationary states and their linear stability, since we expect the formation of aggregates and clusters within the respective groups. Both systems - the one with local jumps and non-local sensing as well as the one with non-local jumps and local sensing - have constant stationary states (given by \eqref{e:stationary}).\\
We start by analysing the linear stability of these constant stationary states and identify conditions for the function $D_{r,b}$, which ensure stability.  Next we are interested in non-trivial stationary states. In the case of a single species, we are able to characterise and approximate non-trivial stationary states in 1D. We will see that the value of the approximate solutions changes quickly from $\underline{r}$ to $\overline{r}$ in certain parameter regimes, where $0<\underline{r}<\overline{r}<1$ will be specified below. All results will be illustrated by numerical experiments in Section \ref{sec:numerics}.

\subsection{Linearized stability analysis of constant steady states}
%
In this section we perform a linear stability analysis of the constant density states of the nonlinear PDE system \eqref{system1} as well as the integro differential system \eqref{systemnonlocal4}. We recall that both systems have constant stationary states given by \eqref{e:stationary} on $\mathbb{T}^N$.
\subsubsection{Local jumps and non-local sensing}\label{lin_stab_mf}
%
We start by studying the nonlinear PDE system \eqref{system1}. First we consider the single species model, which is obtained by setting $b \equiv 0$ in \eqref{system1}, i.e.
\begin{align*}
\partial_t r &=  \nabla \cdot [(1-r)\nabla (D_r(K_2\ast r) r)+r D_r(K_2\ast r) \nabla r],
\end{align*}
and then generalise our computations to two species.
In the following we assume that $K_2$ is an admissible kernel in terms of Definition \ref{def:admissible} and $\int_{\mathbb{R}^N} K_2(x)\, dx=1$.
%
\paragraph*{Single species model}
We consider a perturbation around the constant stationary state $r_0$, given by \eqref{e:stationary}, that is $r(x)=r_0+\epsilon \tilde{r}(x)$, where  $\int_{\mathbb{T}^N} \tilde{r}~ dx=0$. 
Using $D_r(p)=C_r e^{-C_{rr} p}$ and the fact that the mass of $K_2$ is normalised to $1$, we have $K_2\ast r=K_2\ast (r_0+\epsilon \tilde{r})=r_0+\epsilon K_2\ast \tilde{r}$.
This implies
\begin{align*}
D_r(K_2\ast r)&=D_r(K_2\ast (r_0+\epsilon \tilde{r}))=D_r(r_0)-\epsilon\, C_{rr} D_r( r_0) K_2\ast \tilde{r}+\mathcal{O}(\epsilon^2).
\end{align*}
This gives us the following linearized equation (up to order $\epsilon$)
\begin{align}\label{eq:fourier}
\partial_t \tilde{r}&=D_r(r_0) \Delta \tilde{r}-(1-r_0)r_0 C_{rr}D_r( r_0)\Delta (K_2 \ast \tilde{r}).
\end{align}
Then, a Fourier transform yields
\begin{align*}
\partial_t \hat{r}&=-|\xi|^2 \left(D_r(r_0) -(1-r_0)r_0 C_{rr}D_r( r_0) \hat{K_2}(\xi)\right)\hat{r},
\end{align*}
where $\hat{r}=\hat{r}(\xi,t)$ denotes the Fourier transform of $\tilde{r}$. Note that as $K_2$ is even, $\hat{K_2}$ is a real function. Hence only wavenumbers $\xi$, for which 
\begin{align}
\hat{K}_2(\xi)& < \frac{1}{(1-r_0)r_0C_{rr}}\geq 0,
\end{align}
where $0<r_0 <1$, are expected to be stable. Since $K_2\in L^1(\mathbb{R}^N)$, we have that $\hat{K}_2$ is a continuous function vanishing at infinity, i.e. $\hat{K}_2\in C_0(\mathbb{R}^N)$. Furthermore, thanks to the assumptions on admissible kernels, cf. Definition \ref{def:admissible}, $\hat{K}_2$ is a positive, radially symmetric, decreasing function. Hence all wavenumbers $\xi$ with $|\xi|$ larger than some certain threshold are stable.

\paragraph*{Full two species model}
In the case of two species we rewrite system \eqref{system1} as
\begin{align}\label{system2}
\begin{aligned}
\begin{split}
\pa_t r&=\nabla \cdot [(1-b) D_r(K_2\ast r, K_2\ast b)\nabla r+(1-\rho) r \nabla D_r(K_2\ast r,K_2\ast b)\\
&\phantom{= \nabla \cdot}+r D_r(K_2\ast r,K_2\ast b)\nabla b]
\end{split}\\
\begin{split}
\pa_t b &= \nabla \cdot [(1-r) D_b( K_2\ast r, K_2\ast b)\nabla b+(1-\rho) b \nabla D_b(K_2\ast r,K_2\ast b)\\
&\phantom{ = \nabla \cdot} +b D_b(K_2\ast r,K_2\ast b)\nabla r].
\end{split}
\end{aligned}
\end{align}
Once again we use the perturbation ansatz $r(x)=r_0+\epsilon \tilde{r}(x)$ and $b(x)=b_0+\epsilon\tilde{b}(x)$, where $r_0$ and $b_0$ are the constant stationary states defined in \eqref{e:stationary} and $\int_{\mathbb{T}^N} \tilde{r}~ dx=\int_{\mathbb{T}^N} \tilde{b}~dx=0$. \\
We recall that the diffusivities take the form
\begin{align*}
D_r(K_2\ast r, K_2\ast b)=C_r e^{-C_{rr} K_2\ast r+C_{rb} K_2\ast b} \text{ and }D_b(K_2\ast r,K_2\ast b)=C_be^{C_{br} K_2\ast r-C_{bb}K_2\ast b}.
\end{align*}
Then the linearized system reads as
\begin{align}\label{system3}
\begin{aligned}
\pa_t \tilde{r}&=(1-b_0) D_r(r_0,b_0)\Delta \tilde{r}-C_{rr}(1-\rho_0) r_0 D_r( r_0, b_0) \Delta (K_2\ast \tilde{r})\\
&\quad+C_{rb}(1-\rho_0) r_0 D_r( r_0, b_0) \Delta (K_2\ast \tilde{b}))+r_0 D_r( r_0, b_0)\Delta \tilde{b}\\
\pa_t \tilde{b}&=(1-r_0) D_b( r_0,b_0)\Delta \tilde{b}-C_{bb}(1-\rho_0) b_0 D_b( r_0,b_0) \Delta (K_2\ast \tilde{b})\\
&\quad+C_{br}(1-\rho_0) b_0 D_b(r_0, b_0) \Delta (K_2\ast \tilde{r}))+b_0 D_b(r_0, b_0)\Delta \tilde{r}.\\
\end{aligned}
\end{align}
Fourier transform gives 
\begin{align}\label{system4}
\begin{aligned}\begin{pmatrix}
\partial_t \hat{r} \\ 
\partial_t \hat{b}
\end{pmatrix} &=\underbrace{\begin{pmatrix}
L_{11}&L_{12}\\
L_{21} & L_{22}
\end{pmatrix}}_{=:A}
\begin{pmatrix}
\hat{r} \\ 
\hat{b}
\end{pmatrix}
\end{aligned}
\end{align}
where 
\begin{align*}
&L_{11}=-|\xi|^2 D_r( r_0, b_0)[(1-b_0)-C_{rr}(1-\rho_0) r_0 \hat K_2(\xi)],\\
&L_{12}=-|\xi|^2 D_r( r_0, b_0)[r_0 +C_{rb}(1-\rho_0) r_0 \hat K_2(\xi)],\\
&L_{21}=-|\xi|^2 D_b( r_0, b_0)[b_0 +C_{br}(1-\rho_0) b_0  \hat K_2(\xi)], \\
&L_{22}=-|\xi|^2 D_b( r_0, b_0)[(1-r_0) -C_{bb}(1-\rho_0) b_0  \hat{K}_2(\xi)].
\end{align*}
The linearized system is linearly stable, if all eigenvalues of the matrix $A$ have strictly negative real values. In order to understand the stability or instability depending on the preferences, let us consider some asymptotic cases. First of all, for $C_{ij}$ small (i.e. small preference) we observe that the constant stationary state is linearly stable, since the eigenvalues  of $A$ are a perturbation of 
\begin{multline*}
\lambda_{1,2}=-\frac{|\xi|^2}{2}\bigg(D_r(1-b_0)+D_b(1-r_0) \\
   \pm \sqrt{[( D_r(1-b_0) + D_b (1-r_0))^2 - 4 (D_b D_r (1-b_0-r_0))}\bigg)
\end{multline*}
   in the case $C_{ij}=0$. Hence, the mean-field model does not lead to segregation for arbitrarily small preference, but a certain threshold is needed. On the other hand for large $C_{ij}$ the constant stationary state becomes unstable. The easiest case to see this is 
$C_{ij} = C$ and $C \rightarrow \infty$. In this case the eigenvalues are a small perturbation of $\lambda_1=0$ and $\lambda_2=|\xi|^2\hat{K}_2(\xi)(1-\rho_0)(D_r r_0+D_b b_0)$. Note that we assumed $K_2$ to be positive, radially symmetric and non-increasing implying $\hat{K}_2(\xi) > 0$ for all $\xi\in \mathbb{R}^N$.

\subsubsection{Non-local jumps and local sensing}\label{lin_stab_int} 

The linear stability analysis for the integro differential equation is similar to the one of the mean field PDE system. In the following we assume that $K_1$ is an admissible kernel in the sense of Definition \ref{def:admissible} and $\int_{\mathbb{R}^N} K_1(x)\, dx=M_{K_1}$.\\

We use the same perturbation ansatz as in the previous subsections, namely $r(x)=r_0+\epsilon\tilde{r}(x)$ and $b(x)=b_0+\epsilon\tilde{b}(x)$ with $\int_{\mathbb{T}^N} \tilde{r}\, dx=\int_{\mathbb{T}^N} \tilde{b}\, dx =0$. This gives the following linearized system of equations
\begin{align*}
\partial_t \tilde{r} &= (1-\rho_0)K_1*[(\partial_p D_r(r_0,b_0)\tilde{r}+\partial_q D_r(r_0,b_0) \tilde{b})r_0+D_r(r_0,b_0)\tilde{r}]-\tilde{\rho}K_1*(D_r(r_0,b_0)r_0)\\
&\quad+D_r(r_0,b_0)r_0K_1*\tilde{\rho}-(\partial_p D_r(r_0,b_0)\tilde{r}+\partial_q D_r(r_0,b_0)\tilde{b})r_0K_1* (1-\rho_0)\\
&\quad-D_r(r_0,b_0)\tilde{r}K_1*(1-\rho_0)\\
\partial_t \tilde{b} &=(1-\rho_0)K_1*[(\partial_p D_b(r_0,b_0)\tilde{b}+\partial_q D_b(r_0,b_0) \tilde{b})b_0+D_b(r_0,b_0)\tilde{b}]-\tilde{\rho}K_1*(D_b(r_0,b_0)b_0)\\
&\quad+D_b(r_0,b_0)b_0K_1*\tilde{\rho}-(\partial_p D_b(r_0,b_0)\tilde{b}+\partial_q D_b(r_0,b_0)\tilde{b})b_0K_1* (1-\rho_0)\\
&\quad-D_b(r_0,b_0)\tilde{b}K_1*(1-\rho_0),
\end{align*}
which can be written as
\begin{align*}
\partial_t \tilde{r} &= (1-\rho_0)r_0(\partial_p D_r(r_0,b_0)K_1*\tilde{r}+\partial_q D_r(r_0,b_0) K_1*\tilde{b})+(1-\rho_0)D_r(r_0,b_0)K_1*\tilde{r}\\
&\quad-M_{K_1}\tilde{\rho}D_r(r_0,b_0)r_0 +D_r(r_0,b_0)r_0 K_1*\tilde{\rho}-M_{K_1}(\partial_p D_r(r_0,b_0)\tilde{r}+\partial_q D_r(r_0,b_0)\tilde{b})r_0 (1-\rho_0)\\
&\quad-M_{K_1}D_r(r_0,b_0)\tilde{r}(1-\rho_0)\\
\partial_t \tilde{b} &=(1-\rho_0)b_0(\partial_p D_b(r_0,b_0)K_1*\tilde{r}+\partial_q D_b(r_0,b_0) K_1*\tilde{b})+(1-\rho_0)D_b(r_0,b_0)K_1*\tilde{b}\\
&\quad-M_{K_1}\tilde{\rho}D_b(r_0,b_0)b_0 +D_b(r_0,b_0)b_0 K_1*\tilde{\rho}-M_{K_1}(\partial_p D_b(r_0,b_0)\tilde{r}+\partial_q D_b(r_0,b_0)\tilde{b})r_0 (1-\rho_0)\\
&\quad-M_{K_1}D_b(r_0,b_0)\tilde{b}(1-\rho_0).
\end{align*}
Using Fourier transform, we obtain a linear system of the form
\begin{align*}
\partial_t \hat{r} &=[(1-\rho_0)r_0 \partial_p D_r(r_0,b_0)+(1-b_0)D_r(r_0,b_0)](\hat{K}_1-M_{K_1})\hat{r}\\
&\quad +[(1-\rho_0)r_0 \partial_q D_r(r_0,b_0)+D_r(r_0,b_0)r_0](\hat{K}_1-M_{K_1})\hat{b}\\
\partial_t \tilde{b} &=[(1-\rho_0)b_0 \partial_p D_b(r_0,b_0)+(1-r_0)D_b(r_0,b_0)](\hat{K}_1-M_{K_1})\hat{b}\\
&\quad +[(1-\rho_0)b_0 \partial_q D_b(r_0,b_0)+D_b(r_0,b_0)b_0](\hat{K}_1-M_{K_1})\hat{r},\\
\end{align*}
which reads in a more compact form as
\begin{align}\label{system6}
\begin{aligned}\begin{pmatrix}
\partial_t \hat{r} \\ 
\partial_t \hat{b}
\end{pmatrix}
&=\underbrace{\begin{pmatrix}
L_{11}&L_{12}\\
L_{21} & L_{22}
\end{pmatrix}}_{=:A}
\begin{pmatrix}
\hat{r} \\ 
\hat{b}
\end{pmatrix}
\end{aligned}
\end{align}
where 
\begin{align*}
&L_{11}=D_r( r_0, b_0)(\hat{K}_1-M_{K_1})[(1-b_0)-C_{rr}(1-\rho_0)r_0],\\
&L_{12}=D_r( r_0, b_0)(\hat{K}_1-M_{K_1})[r_0+C_{rb}(1-\rho_0)r_0],\\
&L_{21}=D_b( r_0, b_0)(\hat{K}_1-M_{K_1})[b_0+C_{br}(1-\rho_0)b_0],\\
&L_{22}=D_b( r_0, b_0)(\hat{K}_1-M_{K_1})[(1-r_0)-C_{bb}(1-\rho_0)b_0].
\end{align*}

Analogously to Section \ref{lin_stab_mf}, we investivate the impact of small and large preferences. Note that as we assumed $K_1$ to be admissible in the sense of Definition \ref{def:admissible}, we can deduce that $\hat{K}_1(\xi)-M_{K_1}\leq 0$ for all $\xi \in \mathbb{R}^N$ and $\hat{K}_1(\xi)=M_{K_1}$ if and only if $\xi=0$. Hence, for $\xi \neq 0$, we obtain the same eigenvalues for the limit cases $C_{ij}=0$ and $C_{i,j}\to \infty$ as in the mean-field PDE model.
However, the stability of solutions does not depend on the frequency of the perturbations (as in the mean-field PDE model) and we will observe in Section \ref{sec:numerics} that the formation
of aggregates happens at a much faster time scale. These observations are based on numerical experiments only, as we are not able to compare the dynamics of the
linearized systems analytically at the moment.

\subsection{Local jumps and non-local sensing: nontrivial steady states of the single-species PDE model}
We conclude by analysing non-constant stationary states of the single species model. We recall that the single species model is obtained by setting $b \equiv 0$ in \eqref{system1}. Since the diffusivity decreases with the perceived density
we expect the formation of aggregated states, so called bumps later on.
The single species model reads as 
\begin{align}
\partial_t r &=  \nabla \cdot [(1-r)\nabla (D_r(K_2\ast r) r)+r D_r(K_2\ast r) \nabla r]\nonumber\\
&=\nabla \cdot [D_r(K_2\ast r) \nabla r +(1-r) r\nabla D_r(K_2\ast r)],\label{1D_reduction}
\end{align}
with $D_r(K_2\ast r)=C_re^{-C_{rr} K_2\ast r}$ and $K_2$ an admissible kernel specified below. 
We set the flux  $J:=D_r (K_2 \ast r)\nabla r +(1-r) r\nabla D_r$ to zero to identify possible stationary states:
\begin{align}\label{eq:ss} 
\frac{\nabla r}{(1-r)r}+\frac{\nabla D_r(K_2 \ast r)}{D_r(K_2 \ast r)}&=0,
\end{align}
where we exclude the values $r(x)=0$ and $r(x)=1$, which will be justified by the following computations. Equation \eqref{eq:ss} results in
\begin{align}\label{line4}
r &=\frac{\tilde{C}}{\tilde{C}+D_r(K_2 \ast r)},
\end{align}
where $\tilde{C}>0$ is a constant of integration. Equation \eqref{line4} is an integral equation, which we want to approximate by a differential equation in the following.

From now on we consider equation \eqref{line4} on the one dimensional torus $\mathbb{T}^1$, which can be interpreted as the interval $[0,1]$ with periodic boundary conditions. For $K_2$ we choose $K_2(x)=\frac{1}{2\epsilon}e^{-\frac{|x|}{\epsilon}}$, which can be approximated (up to exponentially small terms) on the torus by $\tilde{K}_2(x)=\frac{1}{2\epsilon}e^{-\frac{d(x)}{\epsilon}}$ with $d(x)=\min\{|x|,1-|x|\}$, i.e. $(\tilde{K}_2\ast r)(x)=\int_0^1 \tilde{K}_2(x-y)r(y)\, dy$ and, thus,
\begin{align*}
  (\tilde{K}_2\ast r)(x) &=\begin{cases}
\frac{1}{2\epsilon}\int_0^{x+\frac{1}{2}} e^{-\frac{|x-y|}{\epsilon}}r(y)\, dy + \frac{1}{2\epsilon}\int_{x+\frac{1}{2}}^1 e^{\frac{|x-y|-1}{\epsilon}}r(y)\, dy\quad \text{for } x\le \frac{1}{2} \\
\frac{1}{2\epsilon}\int_0^{x-\frac{1}{2}} e^{\frac{|x-y|-1}{\epsilon}}r(y)\, dy + \frac{1}{2\epsilon}\int_{x-\frac{1}{2}}^1 e^{-\frac{|x-y|}{\epsilon}}r(y)\, dy\quad \text{for } x>\frac{1}{2} .
\end{cases}
\end{align*}
In the following we assume that  $\epsilon\ll 1$, hence the diffusivity depends locally on the density. For $x<\frac{1}{2}$ we introduce the new variables $y=x+\epsilon z$ and $y=x+1+\epsilon z$ in the first and second integral, respectively. This gives 
\begin{align*}
(\tilde{K}_2\ast r)(x) &=\frac{1}{2}\int_{-\frac{x}{\epsilon}}^{\frac{1}{2\epsilon}} e^{-|z|}r(x+\epsilon z)\, dz + \frac{1}{2}\int_{-\frac{1}{2\epsilon}}^{-\frac{x}{\epsilon}} e^{z}r(x+1+\epsilon z)\, dz\\
&\sim \frac{1}{2}\int_{-\infty}^\infty  e^{-|z|} (r(x)+\epsilon r'(x) z +\frac{\epsilon^2 z^2}{2}r''(x)+\mathcal{O}(\epsilon^3))\, dz\\
&= r(x)+\epsilon^2 r''(x)+\mathcal{O}(\epsilon^3).
\end{align*}
For the case $x>\frac{1}{2}$, we also obtain in an analogous way that $\tilde{K}_2\ast r\sim r+\epsilon^2 r''+\mathcal{O}(\epsilon^3)$.
Since $\frac{rD_r(\tilde{K}_2 \ast r)}{1-r}=\tilde{C}$ with $D_r(f)=C_r e^{-C_{rr} f}$, we obtain the approximation
\begin{align}\label{ode1}
\epsilon^2 r''+g(r)&=0 \,,\qquad\mbox{with } 
g(r)=r+\frac{1}{C_{rr}}\log \left(\frac{\tilde{C}}{C_r}\left(\frac{1}{r}-1\right)\right) \,.
\end{align}
For $C_{rr}\le 4$, $g(r)$ is a decreasing function of $r$, implying
that the only periodic solutions of \eqref{ode1} are constant. We therefore
assume $C_{rr}>4$ from now on, whence $g$ is increasing between its extrema
$r_\pm =\frac{1}{2}\pm \sqrt{\frac{1}{4}-\frac{1}{C_{rr}}}$.

Since the evolution conserves the total mass $M = \int_0^1 r\,dx$, we look for solutions oscillating around this value and therefore assume $r_-<M<r_+$ and choose the constant of integration such that
$$
  g(r) = r-M + \frac{1}{C_{rr}} \log\left( \frac{M(1-r)}{r(1-M)} \right)
$$
and, thus, $g(M) = 0$, $g'(M) >0$. Note that then there exist two more zeroes $\underline r$ and $\overline r$ of $g$ with
$0<\underline r < r_-$ and $r_+ < \overline r < 1$.

With $\epsilon$ as bifurcation parameter it is a classical result, cf. \cite{guckenheimer2013nonlinear} Section 3.4,
that steady state bifurcations away from the trivial steady state
$r_0=M$ occur whenever
\begin{align}\label{per_sol2}
\frac{\sqrt{g'(M)}}{\epsilon}=2k\pi \,,\qquad k\ge 1 \,.
\end{align}
The bifurcating solutions have the approximations
$$
  r_k(x) \approx M + a\sin(2k\pi(x-x_0)) \,,
$$
with an appropriate amplitude $a$ and an arbitrary shift $x_0$
(which is due to the translation invariance of the problem).
We expect that the first bifurcation ($k=1$) is transcritical, 
i.e. an exchange of stability between $r_0$ and $r_1$, whereas the bifurcations with $k>1$ produce unstable solutions $r_k$.

Thus, for small $\epsilon$ we expect convergence to $r_1$, which
then (far from the bifurcation) has the approximate form of one plateau with sharp transitions between the values $\underline r$
and $\overline r$. This is confirmed by numerical simulations 
(see Fig. \ref{fig4}). Note that performing analogous steps in the two species model results in a system of equations of second order, which is by far not trivial to analyse and beyond the goal of this paper.

\section{Numerical examples}\label{sec:numerics}

In this section we illustrate the dynamics for both models with various numerical examples on the torus $\mathbb{T}^N$ for $N=1,2$. All simulations are based on an explicit in time stepping. The spatial derivatives are approximated by finite difference quotients, the integrals using the trapezoidal rule. All simulations were implemented and performed in Matlab.\\

In order to compare the two different schemes, the derivation in Section \ref{s:model} suggests the following choice of the kernels:
\begin{align}
K_1(x)&= \frac{1}{\epsilon^{N+2}}\tilde{K}\left( \frac{x}{\epsilon} \right)=\frac{1}{\epsilon^{N+2}}\tilde{k}\left(\left|\frac{x}{\epsilon}\right| \right),\\
K_2(x)&=\frac{1}{\epsilon^{N}}\tilde{K}\left( \frac{x}{\epsilon} \right)=\frac{1}{\epsilon^{N}}\tilde{k}\left(\left|\frac{x}{\epsilon}\right| \right),
\end{align}
with $\int_{\mathbb{R}^N} \tilde{K}(z)|z|^2\, \d z=2N$ as well as $\int_{\mathbb{R}^N} \tilde{K}(z)\, \d z=1$.\\ 
Without loss of generality, we set $\tilde{K}(x)=C_1 e^{-C_2 |x|}$ with $C_1,C_2>0$.
For this choice of the kernel, the constants $C_1$ and $C_2$ can easily be computed using the preceding assumptions and read as
\begin{align}
C_1&=\frac{C_2^N}{|\mathcal{S}^{N-1}|(N-1)!}\text{ and } C_2=\sqrt{\frac{N+1}{2}}.
\end{align}
Note that for $\epsilon \to 0$, the second moment of $K_2$ goes to zero. Hence $K_2$ converges to a Delta Dirac.
Moreover, we assume that the diffusion coefficients have the form \eqref{diff_coeff}.

\subsection{Local jumps and nonlocal sensing}

\subsubsection{One-dimensional case}

We start with 1D simulations for the mean field model. We consider one species and recall that the Fourier transform of 
$K_2(x)=\frac{1}{\epsilon}\tilde{K}\left(\frac{x}{\epsilon}\right)=\frac{C_1}{\epsilon}e^{-\frac{C_2}{\epsilon}|x|}$ is given by
\[\hat{K}_2(\xi)=\frac{2C_1 C_2}{\epsilon^2 4\pi^2\xi^2+C_2^2}.\]
As we are in one dimension, the constants are given by $C_1=\frac{1}{2}$ and $C_2=1$.
Hence, all wavenumbers $\xi$ for which the function
\begin{align}
\begin{aligned}\label{max_instab}
f(\xi)&:=-\xi^2 (1-C_{rr}(1-r_0)r_0\hat{K}_2(\xi))\\
&=-\xi^2 \left(1-C_{rr}(1-r_0)r_0\frac{1}{\epsilon^2 4\pi^2\xi^2+1}\right).
\end{aligned}
\end{align}
is positive, are unstable. Note that the dominant unstable mode has to satisfy $f'(\xi_u)=0$, if $f(\xi_u)>0$. In particular, $\xi_u$ is given by

\begin{align} \label{unstable_mode}
    \xi_u=\frac{\sqrt{\sqrt{C_{rr}(1-r_0)r_0}-1}}{2\epsilon\pi},
\end{align}

where the square root is well defined if and only if $f(\xi_u)>0$.
In the following we will discuss the linear stability of stationary states for two different parameter sets. From \eqref{max_instab} we know that if
\begin{align}\label{stability_condition}
    \hat{K}_2(\xi)=\frac{1}{\epsilon^2 4\pi^2\xi^2+1}<\frac{1}{(1-r_0)r_0C_{rr}},
\end{align}
the stationary solution $r_0$ is stable.

\begin{enumerate}[label=Ex \Roman*)]
\item $~$Let $\epsilon=0.05$, $C_{rr} = 2$ and $r_0 = 0.3$ denote the stationary state. Then inequality \eqref{stability_condition} is satisfied for all wavenumbers $\xi\in \mathbb{R}$, so $r_0$ is linearly stable.
\item $~$Let $\epsilon=0.05$ and $C_{rr} = 10$ and $r_0=0.3$. Then the stability condition \eqref{stability_condition} is satisfied for all wavenumbers $\xi$ with $\xi> \sim 3.34$. 
\end{enumerate}

The corresponding numerical simulations are shown in Figure \ref{fig:stab}. In both simulations we divide the domain in $100$ intervals and use time steps of size $\Delta t= 10^{-4}$. In the first simulation we use the parameters 
discussed in Ex I and start with a perturbation of the form
\[r=r_0+\epsilon \tilde{r}(x)=0.3+ 0.02 \sin(4\pi x).\]
Figure \ref{fig:stab1} shows that the perturbations are smoothed out and the solution goes back to the stationary solution $r_0$. 
If we increase the parameter $C_{rr}$ as in example Ex II, perturbations of the form
\[r=r_0+\epsilon \tilde{r}(x)=0.3+ 0.02 \sin(6\pi x),\]
are unstable, see Figure \ref{fig:stab2}. If we increase the frequency to
\[r=r_0+\epsilon \tilde{r}(x)=0.3+ 0.02 \sin(8\pi x),\]
the stationary states become stable again, see Figure \ref{fig:stab3}. In this example the most unstable mode is $\xi_u \approx 2.13$, cf. \eqref{unstable_mode}. We can observe this dominant mode in the numerical
simulations. Starting with a random perturbation of the form
\[r=0.3+0.01\text{rand}(0,1)\]
then the observed instabilities have period two, see Figure \ref{fig:stab4}.

\begin{figure}[!tbp]
  \begin{center}
  \subfloat[]{\includegraphics[width=0.8\textwidth]{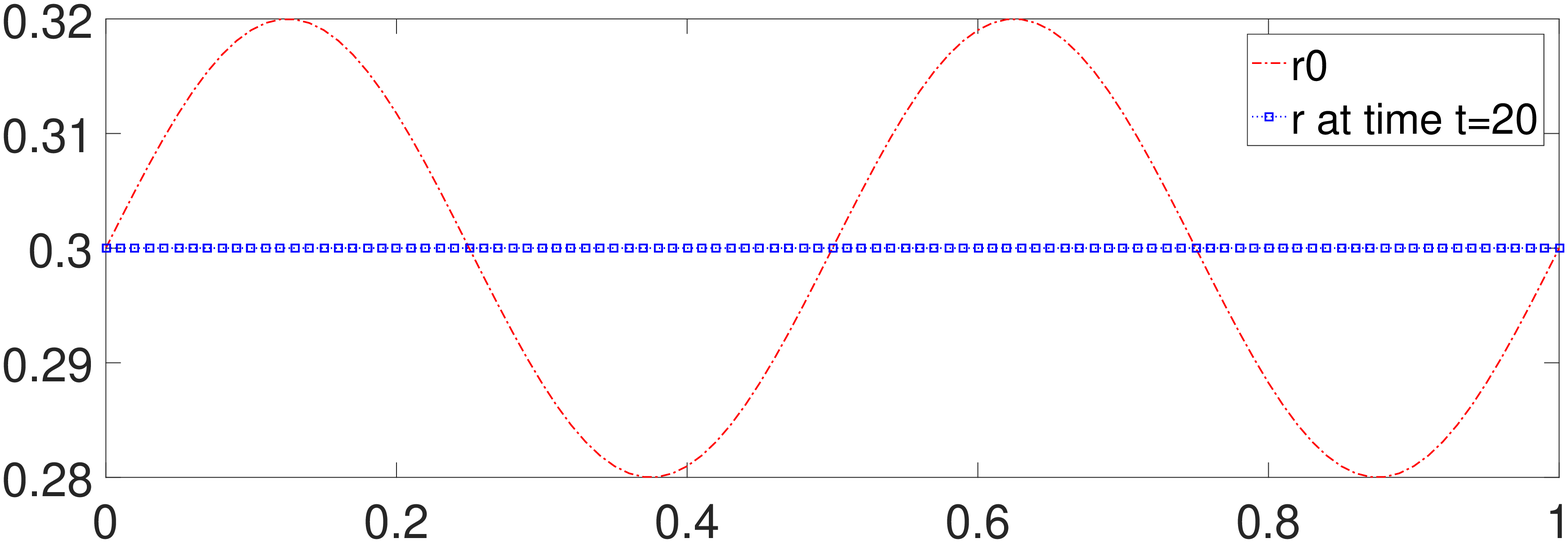}\label{fig:stab1}}\\
  \subfloat[]{\includegraphics[width=0.8\textwidth]{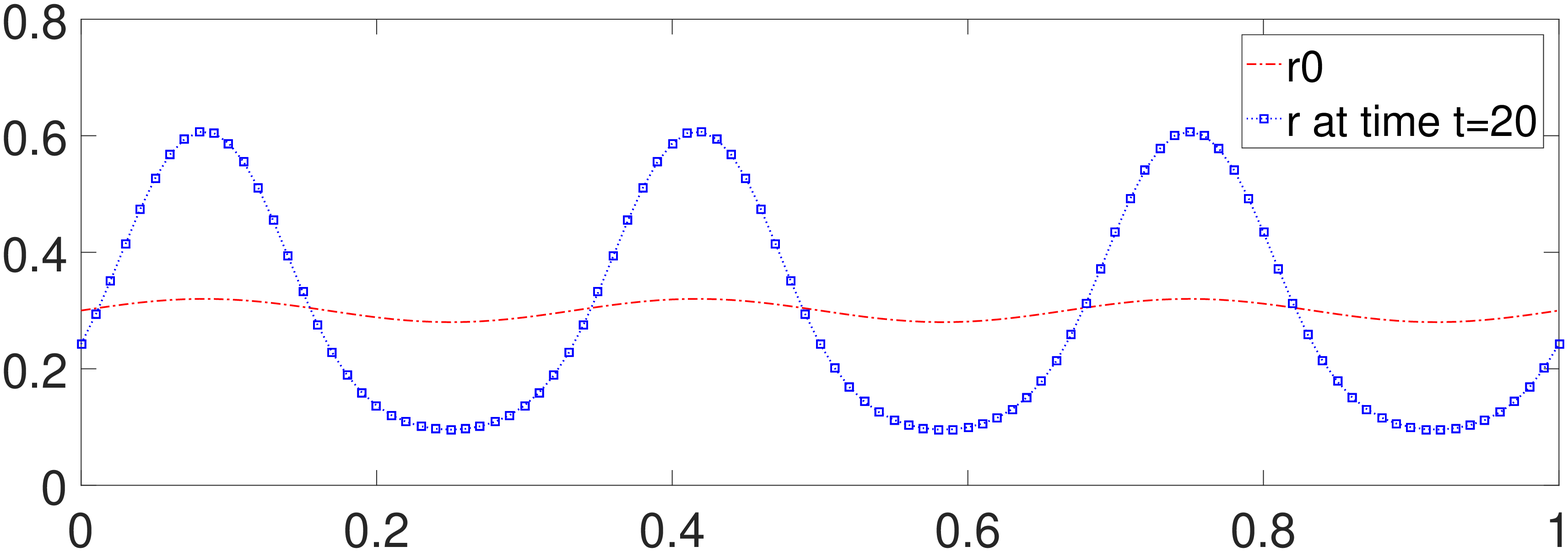}\label{fig:stab2}}\\
  \subfloat[]{\includegraphics[width=0.8\textwidth]{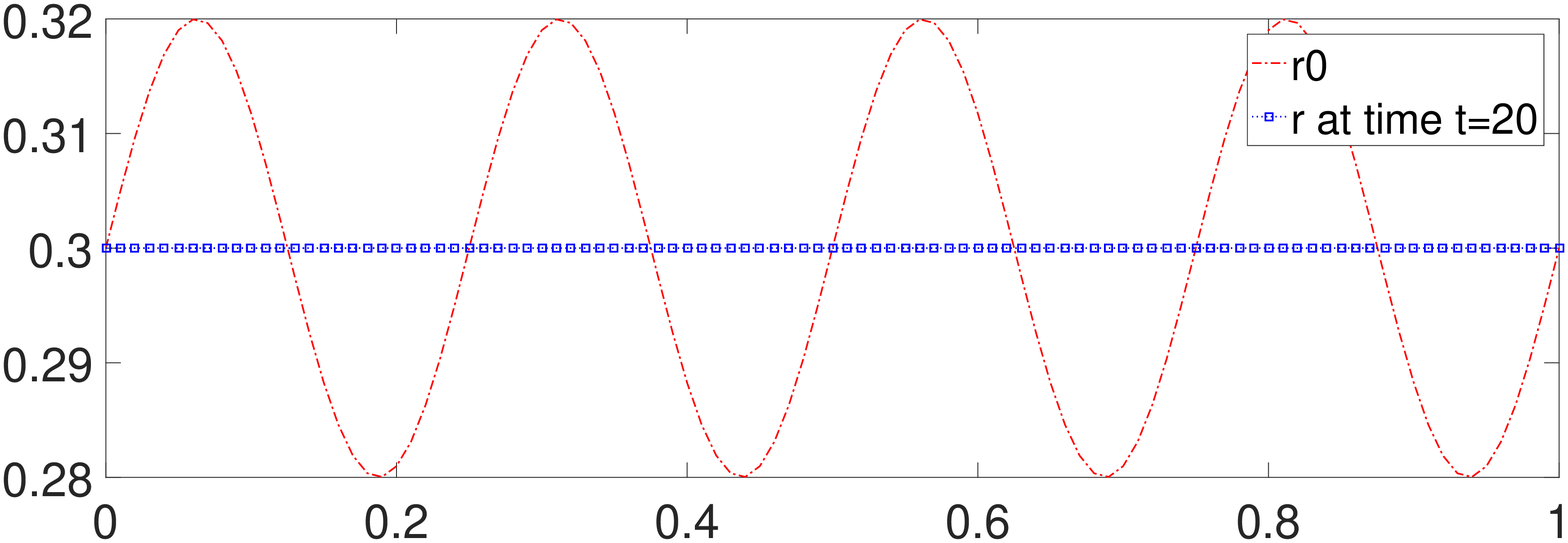}\label{fig:stab3}}\\
  \subfloat[]{\includegraphics[width=0.8\textwidth]{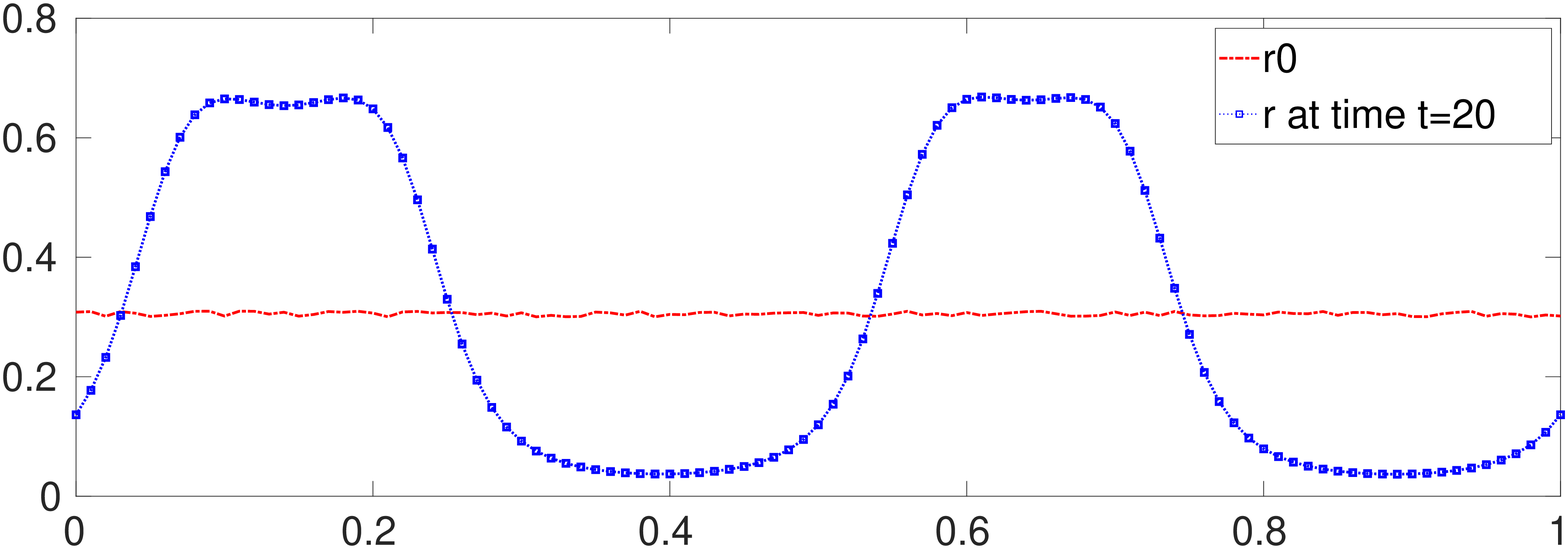}\label{fig:stab4}}
\caption{Evolution of the red particle density in the case of different perturbations. Depending on the magnitude and frequency of the perturbation, the density goes back to constant equilibrium or becomes unstable. }\label{fig:stab}
\end{center}
\end{figure}

\bigskip
Concerning the long time behaviour, we expect that the densities converge to a single aggregate in the long time limit. This type of coarsening dynamics has been observed in similar mean-field systems, see for example \cite{dolak2005keller}. If we start with random initial data in $1D$ and run the simulation for a long time,
we see in Figure \ref{fig4} that the number of bumps decreases in time. We expect a single bump per species as $t \rightarrow \infty$. However these coarsening dynamics are quite hard to resolve, since the convergence becomes exponentially slow.

\begin{figure}[ht]
\begin{center}
\includegraphics[scale=0.23]{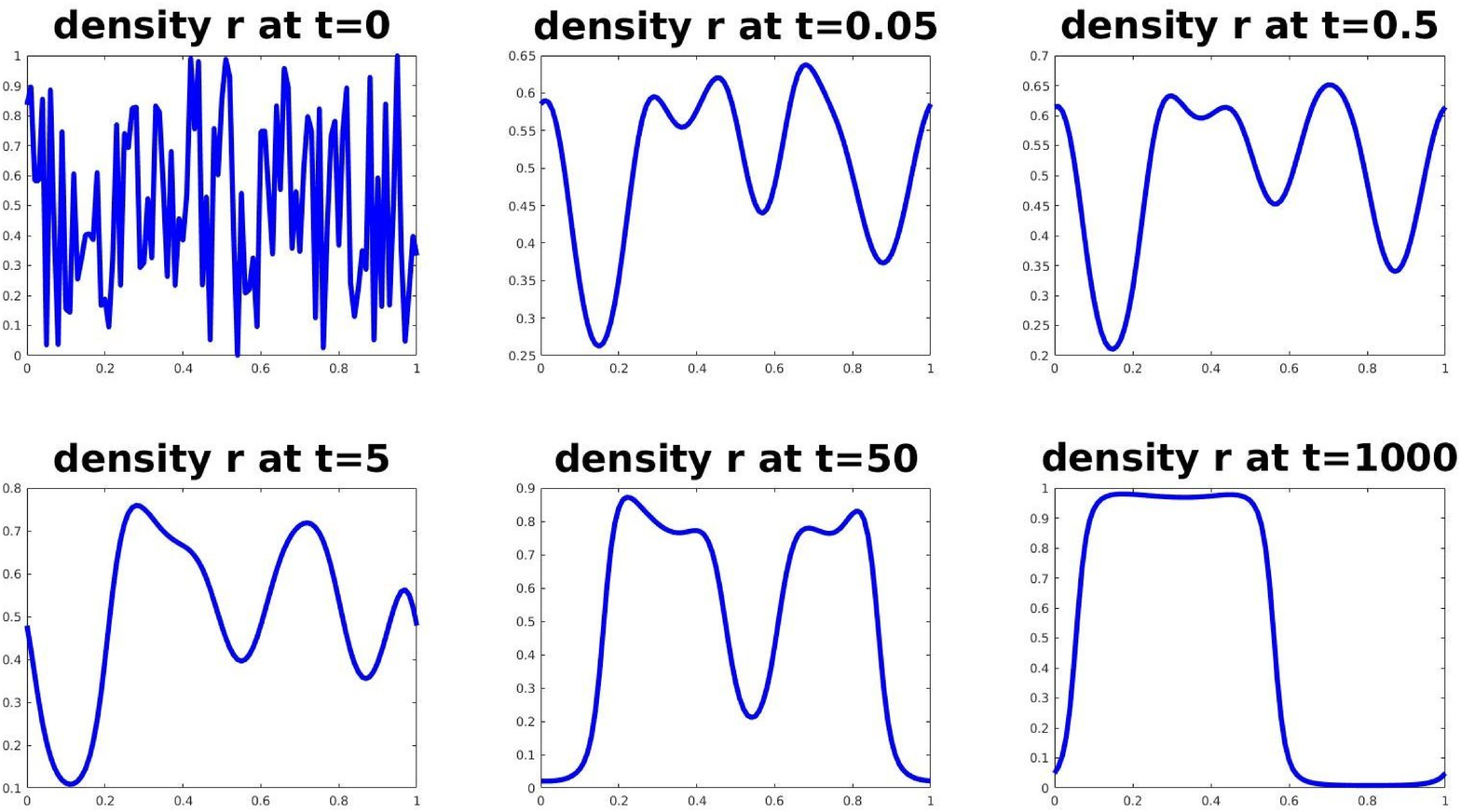}
\end{center}
\caption{Coarsening dynamics in the case of high frequency perturbations in the single species model.}
\label{fig4}
\end{figure}
\bigskip

\subsubsection{Two-dimensional case}

Next we discuss the stability of stationary solutions of the two species model in spatial dimension two, where we divide the domain in a $70\times 70$ grid and choose a time step size of $10^{-4}$.\\
The Fourier transform of $K_2(x)=K_2(x_1,x_2)$ is
\begin{align*}
\hat{K}_2(\xi_1,\xi_2)&= \frac{C_1}{\epsilon^2} \int_{-\infty}^\infty \int_{-\infty}^\infty e^{-\frac{C_2}{\epsilon} \sqrt{x_1^2+x_2^2}}e^{-2\pi i(\xi_1 x_1+\xi_2 x_2)}\, d x_1\, d x_2,
\end{align*}
which gives (after a change to polar coordinates)
\begin{align*}
\hat{K}_2(\mathcal{R}, \Phi )&= \frac{C_1}{\epsilon^2} \int_0^{2\pi}\int_0^\infty e^{-\frac{C_2}{\epsilon} s}e^{-2\pi i \mathcal{R} s \cos(\Phi-\theta)} s\, d s\, d \theta\\
&=\frac{2C_1C_2 \pi}{(\epsilon^2 4\pi^2 \mathcal{R}^2+C_2^2)^{\frac{3}{2}}},
\end{align*}
where $\mathcal{R}^2 =|\xi|^2 =\xi_1^2 +\xi_2^2$. As we are in two dimensions, the constants are given by $C_1=\frac{3}{4\pi}$ and $C_2=\sqrt{\frac{3}{2}}$.

Let $\epsilon=0.1$, $D_r=\frac{1}{10}e^{-10 K_2\ast r+5K_2\ast b}$, $D_b=\frac{1}{10}e^{-10 K_2\ast b+5K_2\ast r}$ and  $r_0=b_0=0.3$ denote a stationary state. 
If we consider a perturbation of the form
\begin{align*}
r(x)=0.3+ 0.02 \sin(4\pi x)\cos(4\pi y)\text{ and }~b(x)=0.3- 0.02 \sin(4\pi x)\cos(4\pi y),
\end{align*}
then one of eigenvalue of the matrix in $A$ defined in equation \eqref{system4} is positive and the expected instabilities arise, see Figure \ref{fig:den2d}. Note that the shape of the arising instabilities comes from the particular choice of the perturbation. 

\begin{figure}[!tbp]
  \centering
  \subfloat[]{\includegraphics[width=4cm, height=5cm]{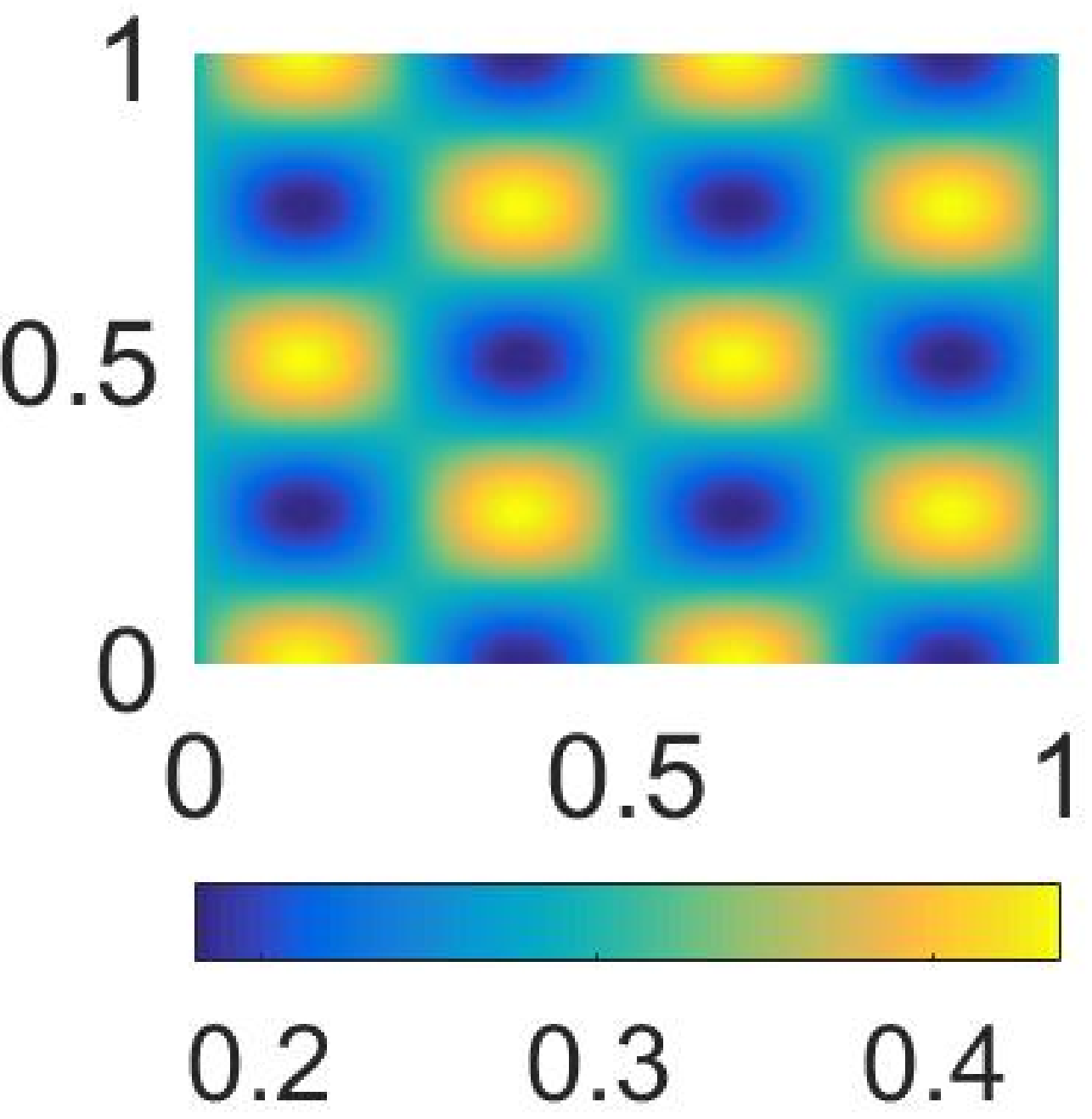}\label{fig:f1}}
  \hspace*{1em}
  \subfloat[]{\includegraphics[width=4cm, height=5cm]{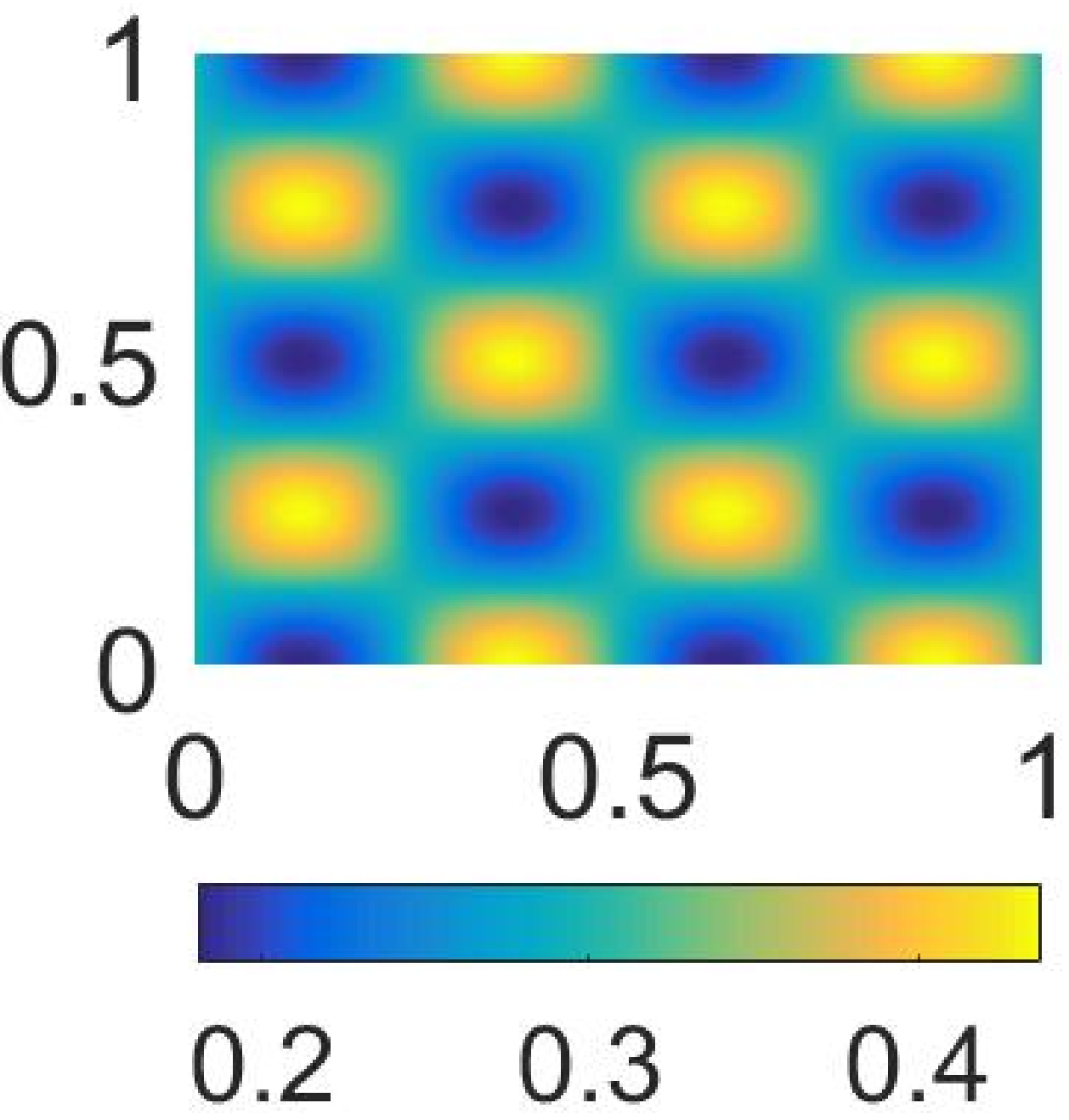}\label{fig:f2}}
\hspace*{1em}
  \subfloat[]{\includegraphics[width=4cm, height=5cm]{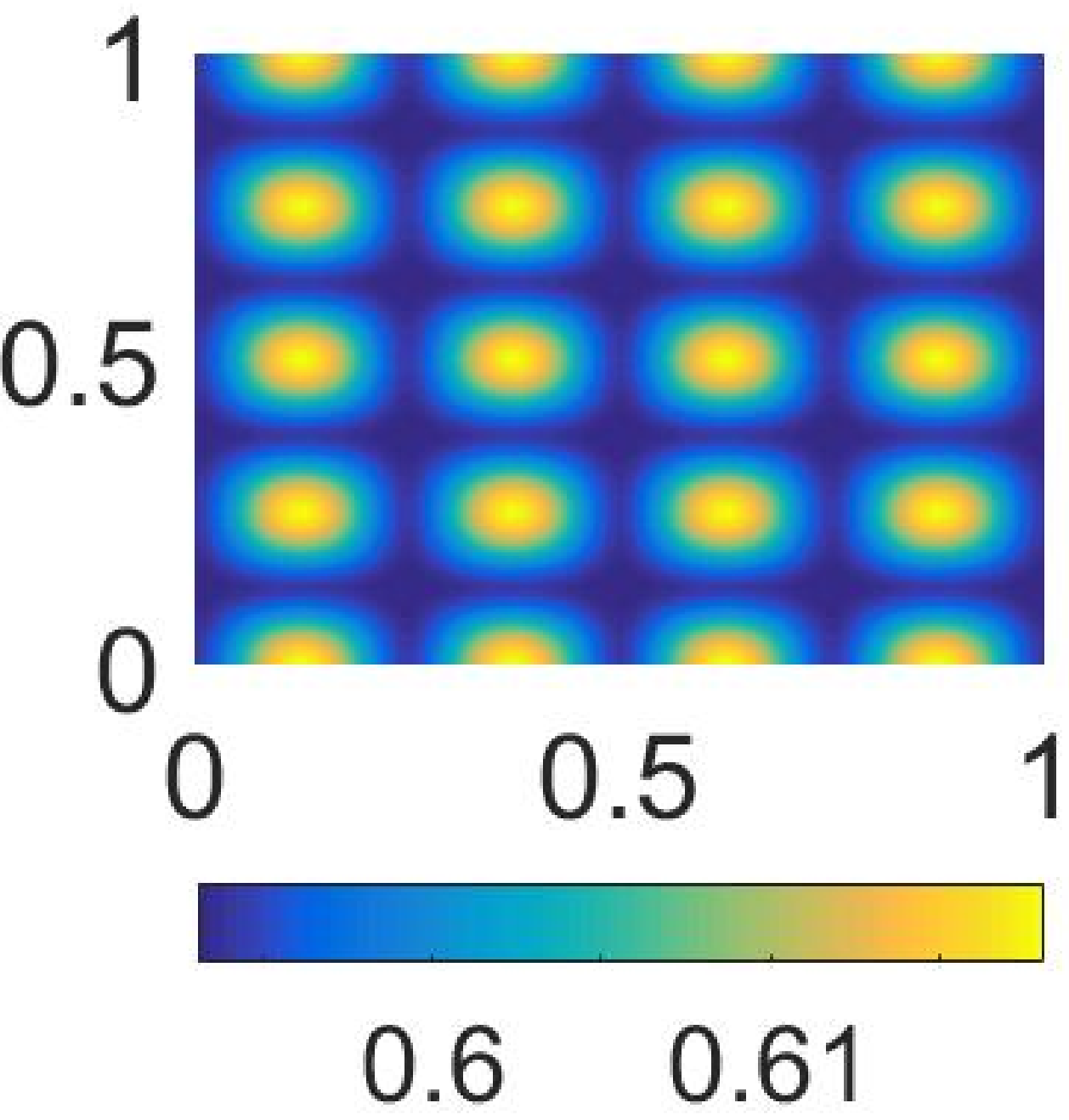}\label{fig:f3}}
\caption{Segregation dynamics in the two species model - the perturbations of the constant equilibrium solutions lead to the formation of aggregates in the densities $r$ and $b$. The plots show the particle densities $r$, $b$ and $\rho$ at time $t=10$.}\label{fig:den2d}
\end{figure}

Increasing the wavenumber $\xi$ in the perturbation, the matrix $A$ from equation \eqref{system4} becomes negative definite and the system is linearly stable, i.e. perturbations smooth out.

Similar to the one dimensional case, we also expect coarsening dynamics in the long time behaviour. However, convergence to non-trivial steady state is exponentially slow.

\subsection{Non-local jumps and local sensing}

For the integro differential equation model, we consider a kernel of the form
\begin{align*}
  K_1(x)=\frac{1}{\epsilon^{N+2}}C_1e^{-C_2\left|\frac{x}{\epsilon}\right|} \text{ with } C_1 = \frac{1}{2} \text{ and } C_2 = 1
  \end{align*}
in 1D. We have seen in Section \ref{lin_stab_int} that wavenumbers $\xi$ for which
\begin{align}
f(\xi)&:=\left(\hat{K}_1-\frac{1}{\epsilon^2}\right)D_r(r_0)[1-C_{rr}(1-r_0)r_0]
\end{align}
is positive, create instabilities. Due to the structure of the kernel, we always have that 
\[\hat{K}_1(\xi)=\frac{1}{\epsilon^2(\epsilon^2 4\pi^2\xi^2+1)}<\frac{1}{\epsilon^2}.\]
Therefore $[1-C_{rr}(1-r_0)r_0]$ has to be negative to observe instabilities. Depending on the choice of $C_{rr}$, the equilibrium solution $r_0=0.3$ is either linearly stable or unstable.  We set $\epsilon=0.05$ and consider a perturbation of the form
\[r=r_0+\epsilon \tilde{r}(x)=0.3+ 0.02 \sin(6\pi x),\]
Figure \ref{fig_integro} shows the result for $C_{rr}=4$ (stable) as well as $C_{rr}=5$ (unstable). In this simulation the domain $\Omega=[0,1]$ was divided into $500$ intervals and the time step set to $\Delta t=10^{-4}$.

\begin{figure}[ht]
\begin{center}
\subfloat{\includegraphics[width=0.8\textwidth]{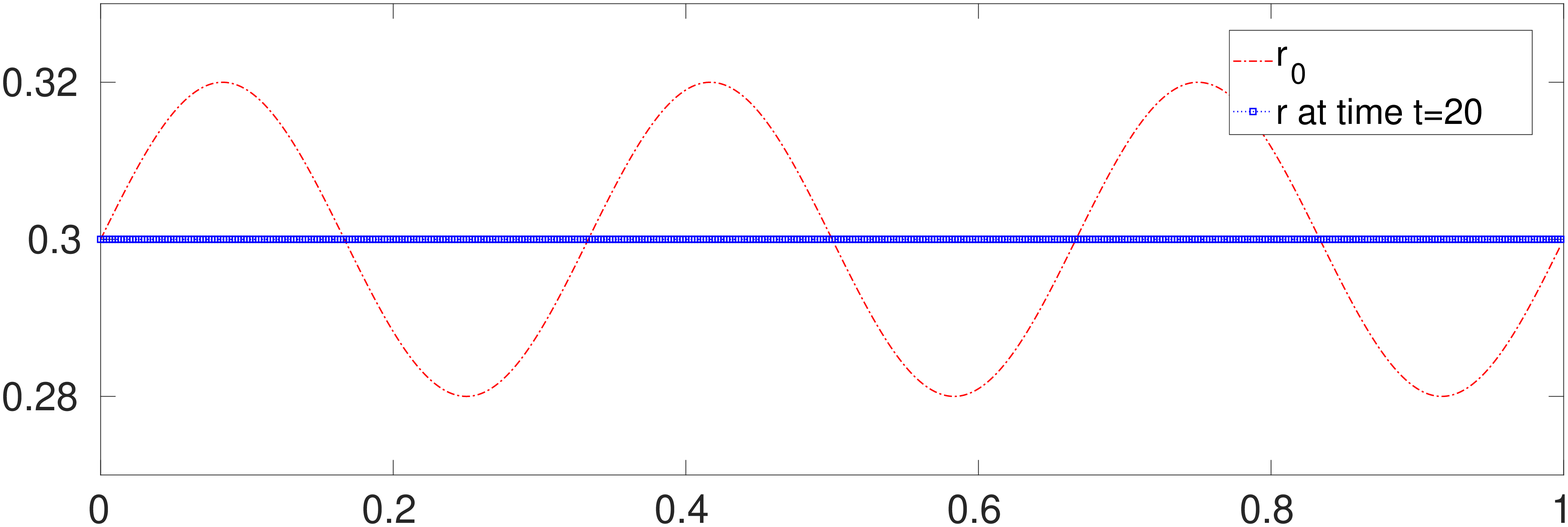}}\\
\subfloat{\includegraphics[width = 0.8\textwidth]{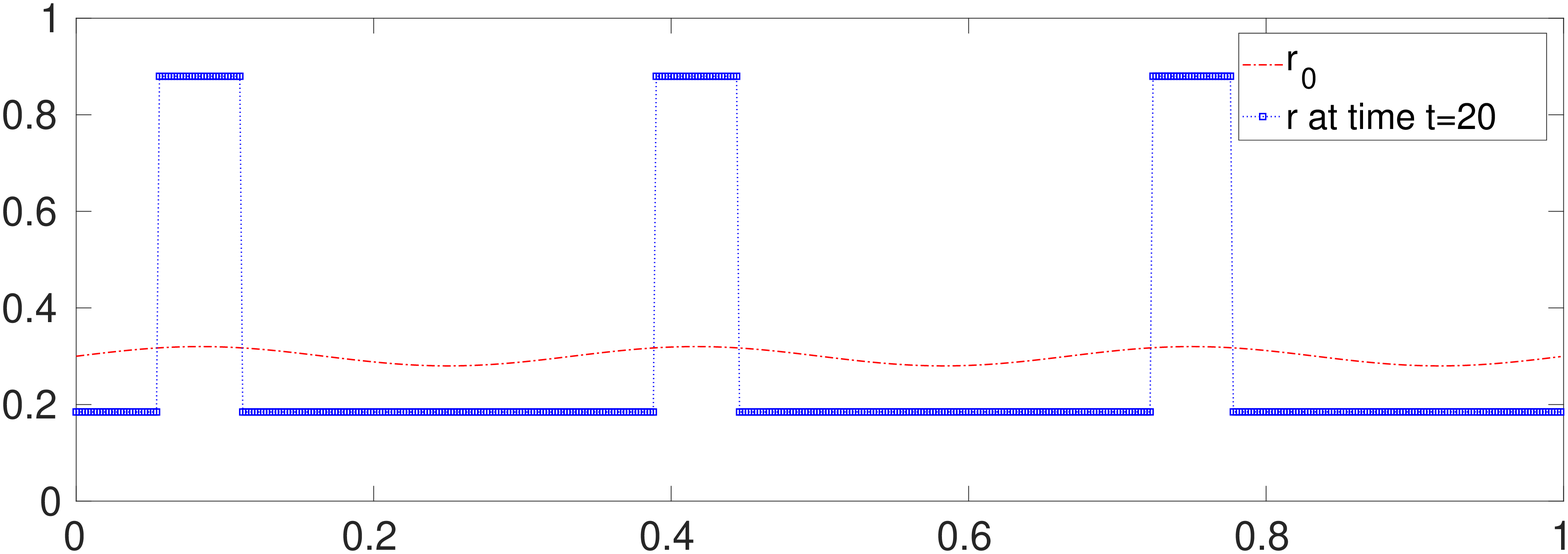}}
\end{center}
\caption{Evolution of the red particle density in the case of different perturbations. Depending on the magnitude and frequency of the perturbation, the density goes back to constant equilibrium or becomes unstable. }
\label{fig_integro}
\end{figure}

We observe a very different behaviour in the two models. While for the PDE model, the frequency of the perturbation plays an essential role concerning the (in)stability regimes, the (in)stability conditions for the integro-differential system do not necessarily depend on it as we have seen in the last example. Moreover, the formation of clusters occurs considerably faster and more intense in the integro-differential model.  

\section{Conclusion}
In this paper we discussed two mean-field models describing the dynamics of individuals belonging to a single or multiple groups, which move randomly in space. In both models the individual diffusivity depends on the locally perceived density - it decreases with the density of the own species and increases with the density of other species present. In the first case individuals move locally but sense the density in a certain region around them. In the second case individuals move globally with a diffusivity depending on the local density only. The preference for the own group leads to the formation of aggregated and segregated stationary states.  We analysed the linear stability of solutions and characterised the stationary states for the single species model in 1D. \\
The presented results serve as a starting point for future research projects. For example the characterisation of stationary states in higher space dimension or the correct resolution of the observed coarsening dynamics is still open. Furthermore the behaviour of solutions for different forms of diffusivities or interaction kernels is of future interest.

\subsection{Acknowledgements}
H.R. acknowledges support by the Austrian Science Fund (FWF) project F 65. The work of J.-F.P. has been supported by DFG via Grant 1073/1-2. M.T.W. acknowledges partial support from the Austrian Academy of Sciences via the New Frontiers Group NST-0001 and the EPSRC via the First Grant EP/P01240X/1. The work of C.S. has been supported by the Austrian Science Fund, Grants no. W1245, SFB 65, and W1261. \\

\noindent The authors thank the referee for checking the manuscript thoroughly, which significantly improved the quality and presentation of the paper.

\bibliographystyle{agsm}
\bibliography{schelling_bib}
\label{lastpage}
\end{document}